\documentclass[a4paper,12pt]{article}

\usepackage{makeidx}

\usepackage{latexsym}
\usepackage{amscd}
\usepackage{amsmath}
\usepackage{amssymb}

\usepackage{amsthm}
\usepackage{float}
\usepackage{graphicx}
\usepackage{psfrag}
\usepackage{pst-node}
\usepackage{pst-plot}

\theoremstyle{plain}
\newtheorem{theorem}{Theorem}[section]
\newtheorem{corollary}[theorem]{Corollary}
\newtheorem{lemma}[theorem]{Lemma}

\newtheorem{proposition}[theorem]{Proposition}

\newtheorem{definition-lemma}[theorem]{Definition-Lemma}

\theoremstyle{definition}
\newtheorem{definition}[theorem]{Definition}

\newcommand{\nin}{\not \in}

\newdimen\argwidth
\def\db[#1\db]{%
  \setbox0=\hbox{$#1$}\argwidth=\wd0
  \setbox0=\hbox{$\left[\box0\right]$}
    \advance\argwidth by -\wd0
  \left[\kern.3\argwidth\box0 \kern.3\argwidth\right]}

\newcommand{\Ob}{\mathfrak{Ob}}

\newcommand{\id}{\operatorname{id}}

\newcommand{\vspan}{\operatorname{span}}

\newcommand{\xto}{\xrightarrow}

\newcommand{\simto}{\xrightarrow{\sim}}

\newcommand{\la}{\left\langle}
\newcommand{\ra}{\right\rangle}

\newcommand{\lc}{\left\{}
\newcommand{\rc}{\right\}}

\newcommand{\vin}{\rotatebox{90}{$\in$}}

\newcommand{\Ext}{\mathop{\mathrm{Ext}}\nolimits}
\newcommand{\Hom}{\mathop{\mathrm{Hom}}\nolimits}

\newcommand{\Spec}{\operatorname{Spec}}

\newcommand{\bCx}{\bC^{\times}}

\newcommand{\coh}{\operatorname{coh}}

\newcommand{\module}{\operatorname{mod}}

\newcommand{\Fuk}{\mathop{\mathfrak{Fuk}}\nolimits}
\newcommand{\Lag}{\mathop{{\scL}ag}\nolimits}
\newcommand{\Lagtilde}{\mathop{\widetilde{{\scL}ag}}\nolimits}
\newcommand{\m}{\mathfrak{m}}

\newcommand{\bC}{\ensuremath{\mathbb{C}}}

\newcommand{\bR}{\ensuremath{\mathbb{R}}}

\newcommand{\bZ}{\ensuremath{\mathbb{Z}}}

\newcommand{\scA}{\ensuremath{\mathcal{A}}}
\newcommand{\scB}{\ensuremath{\mathcal{B}}}
\newcommand{\scC}{\ensuremath{\mathcal{C}}}

\newcommand{\scI}{\ensuremath{\mathcal{I}}}

\newcommand{\scL}{\ensuremath{\mathcal{L}}}
\newcommand{\scM}{\ensuremath{\mathcal{M}}}

\newcommand{\scO}{\ensuremath{\mathcal{O}}}

\newcommand{\scS}{\ensuremath{\mathcal{S}}}

\newcommand{\frakm}{\ensuremath{\mathfrak{m}}}

\newcommand{\stilde}{\widetilde{s}}
\newcommand{\vtilde}{{\tilde{v}}}

\newcommand{\Gtilde}{\widetilde{G}}

\newcommand{\Mtilde}{\widetilde{M}}

\newcommand{\gammatilde}{{\widetilde{\gamma}}}
\newcommand{\phitilde}{{\widetilde{\phi}}}

\newcommand{\Mbar}{{\overline{M}}}

\newcommand{\scdot}{{\, \cdot}}

\newcommand{\pss}{{p^{\sigma \sigma}}}

\newcommand{\dirA}{{A^\to}}
\newcommand{\dirQ}{{Q^\to}}

\newcommand{\dirscA}{{\scA^\to}}
\newcommand{\dirscB}{{\scB^\to}}
\newcommand{\dirscC}{{\scC^\to}}
\newcommand{\dirscI}{{\scI^\to}}

\newcommand{\dirGamma}{{\Gamma^\to}}

\newcommand{\homA}{\hom_{\scA}}

\newcommand{\mA}{\m^{\scA}}
\newcommand{\mB}{\m^{\scB}}
\newcommand{\mC}{\m^{\scC}}

\newcommand{\bfa}{{\boldsymbol{a}}}
\newcommand{\bfb}{{\boldsymbol{b}}}
\newcommand{\bfc}{{\boldsymbol{c}}}
\newcommand{\bfd}{{\boldsymbol{d}}}
\newcommand{\bfe}{{\boldsymbol{e}}}
\newcommand{\bff}{{\boldsymbol{f}}}

\newcommand{\bfY}{{\boldsymbol{Y}}}
\newcommand{\bfP}{{\boldsymbol{P}}}

\title{Exact Lefschetz fibrations\\
associated with dimer models}
\author{Masahiro Futaki and Kazushi Ueda}
\date{}
\begin{document}

\maketitle

\begin{abstract}
We associate an exact Lefschetz fibration
with a pair of a consistent dimer model
and an internal perfect matching on it,
whose Fukaya category is derived-equivalent
to the category of representations of the directed quiver
with relations associated with the pair.
As a corollary,
we obtain a version of homological mirror symmetry
for two-dimensional toric Fano stacks.
\end{abstract}

\section{Introduction}

A dimer model is a bicolored graph on a 2-torus
$
 T = \bR^2 / \bZ^2
$
encoding the information of a quiver with relations.
It is originally introduced by string theorists
\cite{Franco-Hanany-Martelli-Sparks-Vegh-Wecht_GTTGBT,
Franco-Hanany-Vegh-Wecht-Kennaway_BDQGT,
Franco-Vegh_MSGTDM,
Hanany-Herzog-Vegh_BTEC,
Hanany-Kennaway_DMTD,
Hanany-Vegh}
and studied by mathematicians
from various points of view
\cite{
Bender-Mozgovoy,
Bocklandt_CYAWQP,
Broomhead,
Davison,
Mozgovoy-Reineke,
Ishii-Ueda_08,
Ishii-Ueda_09,
Ishii-Ueda_DMEC,
Mozgovoy,
Nagao_Nakajima,
Nagao_DCSCY,
Nagao_RONCDT,
Stienstra_Mahler_dimers,
Stienstra_dessins,
Stienstra_Chow,
Stienstra_CPA,
Szendroi_NCDT}.
If a dimer model $G$ is consistent,
then the quiver $\Gamma$ with relations associated with $G$
has the following properties
\cite{
Bocklandt_CYAWQP,
Broomhead,
Davison,
Ishii-Ueda_09,
Ishii-Ueda_DMEC,
Mozgovoy-Reineke}:
\begin{itemize}
 \item
The derived category
$
 D^b \module_0 \bC \Gamma
$
of finitely-generated nilpotent representations of $\Gamma$
is Calabi-Yau in the sense that
the Serre functor is a shift functor.
 \item
The moduli space $\scM_\theta$ of stable representations of $\Gamma$
with dimension vector $(1, \dots, 1)$ and
a generic stability parameter $\theta$
is a smooth toric Calabi-Yau 3-fold.
 \item
There is an equivalence
$$
 D^b \coh \scM_\theta \cong D^b \module \bC \Gamma
$$
of triangulated categories
between the derived category of coherent sheaves
on $\scM_\theta$ and
the derived category of finitely generated
$\bC \Gamma$-modules.
This restricts to an equivalence
$$
 D^b \coh_0 \scM_\theta \cong D^b \module_0 \bC \Gamma
$$
where $D^b \coh_0 \scM_\theta$ is
the derived category of coherent sheaves on $\scM_\theta$
supported at the inverse image of the origin
by the natural morphism
$
 \scM_\theta \to \Spec \Gamma(\scO_{\scM_\theta}).
$
\item
A toric divisor on $\scM_\theta$ corresponds
to a perfect matching $D$ on $G$.
A perfect matching $D$ is said to be {\em internal}
if it corresponds to a compact toric divisor of $\scM_\theta$
for some stability parameter $\theta$.
One can associate a directed subquiver $\dirGamma$ of $\Gamma$
with an internal perfect matching on $G$.
 \item
For any two-dimensional toric Fano stack $X$,
there is a pair $(G, D)$
of a consistent dimer model $G$
and a perfect matching $D$ on $G$
such that there is an equivalence
\begin{equation} \label{eq:dirXGamma}
 D^b \coh X \cong D^b \module \bC \dirGamma
\end{equation}
of triangulated categories.
\end{itemize}

We prove the following in this paper:

\begin{theorem} \label{th:main}
For a pair $(G, D)$ of a consistent dimer model $G$
and a perfect matching $D$ on $G$,
there is an exact Lefschetz fibration
$
 p : \scS \to \bC
$
such that
$$
 D^b \Fuk p \cong D^b \module \bC \dirGamma.
$$
\end{theorem}

By combining Theorem \ref{th:main} with
\eqref{eq:dirXGamma},
one obtains an equivalence
$$
 D^b \Fuk p \cong D^b \coh X
$$
of triangulated categories,
which is a version of homological mirror symmetry
for two-dimensional toric Fano stacks.
Homological mirror symmetry is proposed by Kontsevich,
originally for Calabi-Yau manifolds \cite{Kontsevich_HAMS} and
later generalized to Fano manifolds \cite{Kontsevich_ENS98}.
The definition of the Fukaya category of a Lefschetz fibration
is due to Seidel \cite{Seidel_VC, Seidel_PL}.
Homological mirror symmetry for Fano manifolds is studied in,
e.g.
\cite{
Abouzaid_HCRMSTV,
Abouzaid_MHTGHMSTV,
Auroux-Katzarkov-Orlov_WPP,
Auroux-Katzarkov-Orlov_dP,
Fang_HMSTP,
Fang-Liu-Treumann-Zaslow_CMT,
Fang-Liu-Treumann-Zaslow_CCC,
Fang-Liu-Treumann-Zaslow_CCTO,
Kerr_WBMSTS,
Seidel_VC2,
Ueda_HMSTdPS}.
The relation between dimer models and
homological mirror symmetry is discovered
in \cite{Feng-He-Kennaway-Vafa} and
followed up in
\cite{Ueda-Yamazaki_NBTMQ,
Ueda-Yamazaki_BTP,
Ueda-Yamazaki_toricdP}.

Theorem \ref{th:main} implies
homological mirror symmetry for toric Calabi-Yau 3-folds
just as in \cite[Theorem 1.1]{Seidel_suspension}:

\begin{corollary} \label{cor:local_hms}
For a smooth toric Calabi-Yau 3-fold $K$
with a compact toric divisor,
there is an exact symplectic manifold $H$
and a full embedding of triangulated categories
$$
 D^b \coh_0 K \hookrightarrow D^b \Fuk H.
$$
\end{corollary}

The organization of this paper is as follows:
In Section \ref{sc:A-inf_dimer},
we recall basic definitions on dimer models and
$A_\infty$-categories,
and introduce the $A_\infty$-category
associated with a dimer model.
In Section \ref{sc:directed_subcategory},
we study the subcategory of $\scA_G$
associated with an internal perfect matching,
and discuss its relation with the derived category
of modules over the path algebra.
In Section \ref{sc:lefschetz},
we construct an exact Lefschetz fibration
from a dimer model and prove
Theorem \ref{th:main} and Corollary \ref{cor:local_hms}.

{\em Acknowledgment}:
We thank Bernhard Keller for several remarks
which led to a major revision of this paper.
We also thank Alastair Craw for pointing out the reference
\cite{Carqueville-Quintero_Velez},
and Nils Carqueville and Alexander Quintero Velez
for sending an earlier version of
\cite{Carqueville-Quintero_Velez}.
K.~U. thanks Akira Ishii for helpful discussions.
M.~F. is supported by Grant-in-Aid for Young Scientists (No.19.8083).
K.~U. is supported by Grant-in-Aid for Young Scientists (No.18840029).
This work has been done while K.~U. is visiting the University of Oxford,
and he thanks the Mathematical Institute for hospitality
and Engineering and Physical Sciences Research Council
for financial support.

\section{An $A_\infty$-category from a dimer model}
 \label{sc:A-inf_dimer}

We first recall basic definitions on dimer models:
\begin{itemize}
 \item
A {\em dimer model} is a bicolored graph
$G = (B, W, E)$ on an oriented 2-torus $T = \bR^2 / \bZ^2$
which divides $T$ into polygons.
Here $B$ is the set of black nodes,
$W$ is the set of white nodes,
and $E$ is the set of edges.
No edge is allowed to connect nodes
with the same color.
 \item
A {\em quiver} consists of
a finite set $V$ called the set of vertices,
another finite set $A$ called the set of arrows,
and two maps $s, t : A \to V$ called the
source and the target map.
The quiver $Q = (V, A, s, t)$
associated with $G$ is defined
as the dual graph of $G$,
equipped with the orientation so that
the white node is always on the right of an arrow;
the set $V$ of vertices is the set of faces of $G$,
and the set $A$ of arrows can naturally be identified
with the set $E$ of edges of $G$.
 \item
For a cyclic path $p = (a_n, \dots, a_1)$
and an arrow $b$ on a quiver $Q$,
the {\em derivative} of $p$ by $b$ is defined by
$$
 \frac{\partial p}{\partial b}
  = \sum_{i=1}^n \delta_{a_i, b}
     (a_{i-1}, a_{i-2}, \dots, a_1, a_n, a_{n-1}, \dots, a_{i+1}),
$$
where
$$
 \delta_{a, b}
  = \begin{cases}
      1 & a = b, \\
      0 & \text{otherwise}.
    \end{cases}
$$
A {\em potential} $\Phi$ is a linear combination
of cyclic paths on the quiver.
The derivative of $\Phi$ is defined by linearity,
and the two-sided ideal $\scI = (\partial \Phi)$
generated by the derivatives $\partial \Phi / \partial a$
for all arrows $a$ gives relations of the quiver.
 \item
The potential $\Phi$ of the quiver $Q$ 
associated with a dimer model is defined by
$$
 \Phi = \sum_{w \in W} c(w) - \sum_{b \in B} c(b),
$$
where $c(n)$ for $n \in B \sqcup W$
is the minimal cyclic path around $n$.
 \item
A {\em perfect matching} is a subset $D \subset E$
such that for any $n \in B \sqcup W$,
there is a unique edge $e \in D$ adjacent to $n$.
A dimer model is {\em non-degenerate}
if for any edge $e \in E$,
there is a perfect matching $D$ such that $e \in D$.
 \item
Two paths $p$ and $q$ are said to be {\em weakly equivalent}
if $p r$ is equivalent to $q r$
for some other path $r$.
A non-degenerate dimer model is said to be {\em consistent}
if weakly equivalent paths are equivalent.
\end{itemize}

Next we recall the difinition
of an $A_\infty$-category.
For a $\bZ$-graded vector space
$N = \bigoplus_{j \in \bZ} N^j$ and an integer $i$,
the $i$-th shift of $N$ to the left
will be denoted by $N[i]$;
$(N[i])^j  = N^{i+j}$.
An $A_\infty$-category $\scA$ consists of
\begin{itemize}
 \item the set $\Ob(\scA)$ of objects,
 \item for $c_1,\; c_2 \in \Ob(\scA)$,
       a $\bZ$-graded vector space $\hom_\scA(c_1, c_2)$
       called the space of morphisms, and
 \item operations
$$\
 \m_l : \hom_\scA (c_{l-1},c_l) \otimes \dots
          \otimes \hom_\scA (c_0,c_1)
 \longrightarrow \hom_\scA (c_0,c_l)
$$
of degree $2 - l$ for $l=1,2,\ldots$
and $c_i \in \Ob(\scA)$,
$i=0,\dots, l$,
satisfying
the {\em $A_\infty$-relations}
\begin{eqnarray}
 \sum_{i=0}^{l-1} \sum_{j=i+1}^l
  (-1)^{\deg a_1 + \cdots + \deg a_i - i}
  \m_{l+i-j+1}(a_l \otimes \cdots \otimes a_{j+1}
   \otimes 
    \m_{j-i}(a_j \otimes \cdots \otimes a_{i+1})
     \nonumber \\
   \otimes
    a_i \otimes \cdots \otimes a_1 ) = 0,
  \label{eq:A_infty}
\end{eqnarray}
for any positive integer $l$,
any sequence $c_0, \dots, c_l$ of objects of $\scA$,
and any sequence of morphisms
$a_i \in \hom_{\scA}(c_{i-1}, c_i)$
for $i = 1, \dots, l$.
\end{itemize}
A {\em cyclic $A_\infty$-category}
of dimension $d \in \bZ$ is a pair $(\scA, \la \bullet, \bullet \ra)$
of an $A_\infty$-category
and a non-degenerate pairing
$$
 \la \bullet, \bullet \ra :
  \hom(c_2, c_1) \otimes \hom(c_1, c_2) \to \bC[d]
$$
which is both symmetric
$$
 \la x, y \ra + (-1)^{(\deg x - 1)(\deg y - 1)} \la y, x \ra = 0
$$
and cyclic
$$
 \la \m_n (x_n, \dots, x_1), x_{0} \ra
  = (-1)^{(\deg x_n - 1)(\deg x_{n-1} + \cdots + \deg x_0 - n)}
    \la \m_{n-1}(x_{n-1}, \dots, x_0), x_n \ra.
$$

As shown in \cite[Section 8.1]{Kontsevich-Soibelman_SSMDTICT},
one can associate a cyclic $A_\infty$-category of dimension three
with any quiver with potential.
By applying their construction to the quiver with potential
associated with a dimer model,
one obtains the following cyclic $A_\infty$-category:

\begin{definition}
Let $G = (B, W, E)$ be a dimer model
and $\Gamma = (V, A, s, t, \scI)$ be the quiver with relations
associated with $G$.
Then the $A_\infty$-category $\scA$
associated with $G$
is defined as follows:
\begin{itemize}
 \item
The set of objects is the set $V$ of vertices of $\Gamma$.
 \item
For two objects $v$ and $w$ in $\scA$,
the space of morphisms is given by
$$
 \homA^i(v, w) =
  \begin{cases}
   \bC \cdot \id_v & i = 0 \text{ and } v = w, \\
   \vspan \{ a \mid a : w \to v \} & i = 1, \\
   \vspan \{ a^\vee \mid a : v \to w \} & i = 2, \\
  \bC \cdot \id_v^\vee & i = 3 \text{ and } v = w, \\
   0 & \text{otherwise}.
  \end{cases}
$$
 \item
Non-zero $A_\infty$-operations are
$$
 \mA_2(x, \id_v) = \mA_2(\id_w, x) = x
$$
for any $x \in \homA(v, w)$,
$$
 \mA_2(a, a^\vee) = \id_v^\vee
$$
and
$$
 \mA_2(a^\vee, a) = \id_w^\vee
$$
for any arrow $a$ from $v$ to $w$,
$$
 \mA_k(a_1, \dots, a_k) = a_0.
$$
for any cycle $(a_0, \dots, a_k)$ of the quiver
going around a white node, and
$$
 \mA_k(a_1, \dots, a_k) = - a_0.
$$
for any cycle $(a_0, \dots, a_k)$ of the quiver
going around a black node.
 \item
The pairing
$$
 \la \bullet, \bullet \ra :
  \homA(w, v) \otimes \homA(v, w) \to \bC[3]
$$
is defined by
$$
 \la a^\vee, a \ra
  = \la \id_v^\vee, \id_v \ra
  = 1
$$
and zero otherwise.
\end{itemize}
\end{definition}

If a dimer model $G$ is consistent,
then the derived category $D^b \scA$
of the $A_\infty$-category $\scA$
associated with $G$ is equivalent
to the derived category $D^b \module_0 \bC \Gamma$
of nilpotent representations of $\Gamma$:

\begin{proposition} \label{prop:quiver-A-infinity}
For a dimer model $G$,
let $\scA$ be the cyclic $A_\infty$-category
associated with $G$ and
$\Gamma$ be the quiver with relations
associated with $G$.
If $G$ is consistent,
then there is an equivalence
\begin{equation} \label{eq:equiv_A-Gamma}
 D^b \scA \cong D^b \module_0 \bC \Gamma
\end{equation}
of triangulated categories.
\end{proposition}

\noindent
{\em Sketch of proof.}
Let $\scC$ be the full subcategory
of the differential graded enhancement of $D^b \module \bC \Gamma$
consisting of simple modules.
Since $D^b \module_0 \bC \Gamma$ is
the smallest triangulated subcategory of $D^b \module \bC \Gamma$
containing simple modules,
one has an equivalence
$$
 D^b \scC \cong D^b \module_0 \bC \Gamma
$$
by Bondal and Kapranov
\cite[\S 4, Theorem 1]{Bondal-Kapranov_ETC}.
Hence the equivalence \eqref{eq:equiv_A-Gamma}
follows from the quasi-equivalence
$$
 \scC \cong \scA
$$
of $A_\infty$-categories,
which can be shown directly
using homological perturbation theory
\cite{Kadeishvili_ASHAA, Gugenheim-Lambe-Stasheff,
Merkulov_SHAKM, Kontsevich-Soibelman_HMSTF},
or deduced from the case of a directed subcategory
discussed in Section \ref{sc:directed_subcategory}
by noting that $\scA$ is the trivial extension
of its directed subcategory
\cite{Carqueville-Quintero_Velez}.

An alternative approach,
suggested to the authors by Bernhard Keller,
is to use Koszul duality for $A_\infty$-categories
developed by Lef\`{e}vre-Hasegawa \cite{Lefevre-Hasegawa_SAC}
and summarized in \cite{Keller_KDCC}:
As $\scA$ is augmented over the product of copies of $\bC$
indexed by the vertices of $Q$,
the bar construction
$$
 C
  = B \scA
  = \bigoplus_{n = 0}^\infty \scA^{\otimes n}
$$
equipped with the co-differential $\delta : C \to C$
defined by
\begin{eqnarray*}
 \delta(a_l \otimes \cdots \otimes a_1)
  = \sum_{i=0}^{l-1} \sum_{j=i+1}^l
     (-1)^{\deg a_1 + \cdots + \deg a_i - i}
      a_l \otimes \cdots \otimes a_{j+1} \phantom{abcdedf} \\
       \otimes 
      \m_{j-i}(a_j \otimes \cdots \otimes a_{i+1})
       \otimes
      a_i \otimes \cdots \otimes a_1,
\end{eqnarray*}
is a co-augmented differential graded coalgebra,
and each $A_\infty$-module $\scM$ over $\scA$
yields a differential graded comodule
$$
 B \scM
  = \bigoplus_{n = 0}^\infty \scM \otimes \scA^{\otimes n}
$$
over $C$.
The functor $\scM \mapsto B \scM$
induces an equivalence
$$
 D^b \scA \to D_0 \, C
$$
from the bounded derived category of $\scA$
to the full triangulated subcategory $D_0 \, C$
of the coderived category $D C$ of $C$
generated by differential graded comodules
coming from the co-augmentation.

Now let $G$ be the completion
of the differential graded algebra
associated with the quiver $(Q, \Phi)$ with potential
by Ginzburg \cite[Section 1.4]{Ginzburg_CYA}.
Then $G$ is the $\bC$-dual of $C$
as observed in \cite[Section 5.3]{Keller_CYTC},
and $\bC$-duality transforms each differential graded $C$-comodule
into a differential graded $G$-module.
This induces a contravariant equivalence
$$
 \Hom_\bC(\bullet, \bC) : D_0 \, C \to D_0(G)
$$
from $D_0 C$ to the full triangulated subcategory $D_0(G)$
of the derived category $D(G)$ of $G$
containing simple modules.
If the dimer model is consistent,
then $\bC \Gamma$ is a Calabi-Yau 3 algebra and
one has a quasi-isomorphism
$
 \bC \Gamma \simto G
$
by \cite[Theorem 5.3.1]{Ginzburg_CYA},
so that $D_0(G)$ is equivalent to $D^b \module_0 \bC \Gamma$.
One can obtain a covariant equivalence
instead of a contravariant one
by composing with the duality functor
$
 \bR \Hom_G(\bullet, G) : D_0(G) \to D_0(G).
$
\qed

\section{The directed subcategory}
 \label{sc:directed_subcategory}

Let $G = (B, W, E)$ be a dimer model and
$Q = (V, A, s, t)$ be the quiver associated with $G$.
A perfect matching $D$ defines a subquiver
$$
 Q_D = (V, A_D, s_D, t_D)
$$
of $Q$
with the same set of vertices as $Q$ and the set
$$
 A_D = A \setminus D
$$
of arrows consisting of those not in $D \subset E = A$.
The path algebra $\bC Q_D$
is naturally a subalgebra of $\bC Q$,
and the intersection $I_D = \bC Q_D \cap \scI$
of the ideal $\scI$ of relations on $Q$
with $\bC Q_D$ gives an ideal of relations on $Q_D$.
We write the resulting quiver with relations as
$\Gamma_D = (Q_D, \scI_D)$.

Let $\scM_\theta$ be the moduli space of
stable representations of $\Gamma$
with respect to a stability parameter $\theta$
in the sense of King \cite{King}.
If $G$ is non-degenerate,
then $\scM_\theta$ is a smooth toric variety
for a generic $\theta$,
and for any toric divisor in $\scM_\theta$,
there is a perfect matching $D$
such that the divisor is defined as the zero locus
of the arrows dual to edges in $D$
\cite{Ishii-Ueda_08}.
In addition, for any perfect matching $D$,
there is a generic stability parameter $\theta$
such that $D$ corresponds to a toric divisor in $\scM_\theta$
in this way.

A perfect matching $D$ is said to be {\em internal}
if the toric divisor in $\scM_\theta$
corresponding to $D$
for some $\theta$
is compact.
It is easy to see that
a perfect matching $D$ is internal
if and only if $Q_D$ does not have an oriented cycle.


For a total order $>$ on the set $V$ of vertices of $Q$,
the {\em directed subquiver} $\dirQ$ is defined
as the subquiver of $Q$
whose set of vertices is the same as $Q$
and whose set of arrows $A^\to$ is given by
$$
 \dirA = \lc a \in A \mid s(a) < t(a) \rc.
$$
The subquiver $\dirQ$ equipped with
the relations $\dirscI = \bC \dirQ \cap \scI$
will be denoted by $\dirGamma$.
A perfect matching $D$ is said to come from a total order $>$
if $\Gamma_D = \dirGamma$.

\begin{lemma} \label{lm:int_pm_directing}
A perfect matching $D$ of a dimer model $G$ is internal
if and only if it comes from a total order $<$
on the set of faces of $G$.
\end{lemma}

\begin{proof}
It is clear that a perfect matching
coming from a total order on the set of faces is internal.
To show the converse,
assume that a perfect matching $D$ is internal.
The non-existence of oriented cycles ensures
that this quiver defines a partial order
on the set of vertices of the quiver.
Choose any total order $<$ compatible with this partial order.
Then the condition that $D$ is a perfect matching
implies that the arrows contained in $D$ is precisely
the arrows $a$ such that $s(a) > t(a)$;
the path $p(a)$ from $t(a)$ to $s(a)$ which goes
around either node adjacent to the edge dual to $a$
is contained in the subquiver and
induces the order $s(a) > t(a)$.
\end{proof}

For an $A_\infty$-category $\scA$ and a sequence
$\bfY = (Y_1, \dots, Y_m)$ of objects,
the associated {\em directed subcategory}
$\dirscA = \scA(\bfY)$
is the $A_\infty$-subcategory
consisting of $\bfY$
such that the spaces of morphisms are given by
$$
 \hom_{\dirscA} (Y_i, Y_j) =
 \begin{cases}
   \bC \cdot \id_{Y_i} & i = j, \\
   \hom_{\scA} (Y_i, Y_j) & i < j, \\
   0 & \text{otherwise},
 \end{cases}
$$
with the $A_\infty$-operations inherited from $\scA$.
If $\scA$ comes from a dimer model $G$
as in Section \ref{sc:A-inf_dimer} and
$<$ is a total order on the set of vertices of the quiver
associated with $G$,
then we set $\bfY = (Y_1, \dots, Y_m)$ to be
the set of objects of $\scA$,
arranged in the order
inverse to the one induced by $<$.

\begin{proposition} \label{prop:directed_A-infinity-quiver}
Let $G$ be a consistent dimer model and
$>$ be a total order on the set of faces of $G$
giving an internal perfect matching.
Then one has an equivalence
$$
 D^b \dirscA \cong D^b \module \bC \dirGamma
$$
of triangulated categories.
\end{proposition}

Proposition \ref{prop:directed_A-infinity-quiver}
comes from a quasi-equivalence
$$
 \dirscA \cong \dirscC
$$
with the the full subcategory $\dirscC$
of the differential graded enhancement of
$
 D^b \module \bC \dirGamma
$
consisting of simple modules.
This is an easy exercise in homological perturbation theory,
and can also be deduced from \cite[Proposition 2]{Keller_AART}
as discussed in 
\cite{Carqueville-Quintero_Velez}.

\section{An exact Lefschetz fibration from a dimer model}
 \label{sc:lefschetz}

We prove the following in this section,
which together with Proposition \ref{prop:directed_A-infinity-quiver}
immediately implies Theorem \ref{th:main}:

\begin{theorem} \label{th:fuk-A-infinity}
Let $G$ be a consistent dimer model and
$D$ be an internal perfect matching on $G$.
Then there is an exact Lefschetz fibration
$p : X \to \bC$
whose Fukaya category is equivalent
to the directed $A_\infty$-category $\dirscA$
associated with $(G, D)$.
\end{theorem}

An {\em exact Lefschetz fibration} is
a $J$-holomorphic function
$
 p : \scS \to \bC
$
on an exact almost K\"{a}hler manifold $(\scS, \omega, J)$
such that all the critical points are non-degenerate.
We also assume that
$J$ is integrable near the critical points, and
the horizontal lift $\gammatilde_x : [0, 1] \to X$
of a smooth path
$\gamma : [0, 1] \to \bC$
starting at $x \in p^{-1}(\gamma(0))$
is always defined.
A {\em distinguished basis of vanishing cycles} is
a collection $(C_1, \dots, C_m)$ of Lagrangian spheres
in the regular fiber of $p$
which collapse to critical points
by parallel transport along a {\em distinguished set of
vanishing paths}, cf. \cite[Section 16]{Seidel_PL}.
By the Fukaya category of $p$,
we mean the directed subcategory
of the Fukaya category of the regular fiber of $p$
consisting of a distinguished basis of vanishing cycles.
Recall from \cite[Lemma 16.9]{Seidel_PL} that
any collection of framed exact Lagrangian spheres
in an exact symplectic manifold
can be realized as a distinguished basis of vanishing cycles
of an exact Lefschetz fibration.

\begin{lemma} \label{lm:M}
There is a 2-manifold $M$ and a collection $(C_v)_{v \in V}$
of embedded circles on $M$ such that intersections of
$C_v$ and $C_{v'}$ are transverse and
in natural bijection
with arrows between $v$ and $v'$.
\end{lemma}

\begin{proof}
Recall that a {\em ribbon graph} is
a graph together with the choice of a cyclic order
on the set of edges connected to each node.
A graph underlying a dimer model naturally has
a ribbon structure
by first giving the cyclic order
on the set of edges around each node
coming from the orientation of $T$, and
then reversing those around the black nodes.
A ribbon graph determines a 2-manifold $M$
by assigning a disk to each node
and gluing them together
as designated by the ribbon structure.
For each vertex $v$ of $\Gamma$
(i.e. for each face of $G$),
one can associate an immersed circle $C_v$ in $M$,
so that arrows of $\Gamma$ naturally correspond
to intersection points between them.
They do not have self-intersections
since the quiver $\Gamma$ does not have a loop,
i.e. an arrow $a$ such that $s(a) = t(a)$;
if such an arrow exists,
no perfect matching comes from a total order
on the set of vertices of $\Gamma$,
so that there can be no internal perfect matching
by Lemma \ref{lm:int_pm_directing}.
\end{proof}

We let $M$ and $C_v$ denote the 2-manifold and
the embedded circles constructed in the proof of
Lemma \ref{lm:M} henceforth.

\begin{lemma} \label{lem:exact}
The 2-manifold $M$ admits an exact symplectic form $\omega$
such that $C_v$ are exact Lagrangian submanifolds.
\end{lemma}

\begin{proof}
Choose an exact complete K\"{a}hler structure on $M$
and let $\theta$ be a one-form on $M$ such that
$\omega = d \theta$ is the K\"{a}hler form.
Recall that $C_v$ is said to be exact
if $\int_{C_v} \theta = 0$.
If we perturb $C_v$ to $C_v'$,
then Stokes' theorem states that
$$
 \int_{C_v} \theta - \int_{C_v'} \theta
  = \int_D \omega,
$$
where $D$ is the region surrounded by $C_v$ and $C_v'$;
$
 \partial D = C_v - C_v'.
$
Note that both sides of $C_v$ contains a point at infinity,
i.e., a point in $\Mbar \setminus M$
where $\Mbar$ is a compactification of $M$.
It follows that for any $R \in \bR$,
one can choose $C_v'$ such that
$\int_D \omega = R$.
By choosing $R = \int_{C_v} \theta$,
one obtains $\int_{C_v'} \theta = 0$ as desired.
\end{proof}

Let
$$
 p : \scS \to \bC
$$
be an exact Lefschetz fibration such that
$
 p^{-1}(0) \cong M
$
and $(C_v)_{v \in V}$ forms a distinguished basis of vanishing cycles.
To equip the Fukaya category with a $\bZ$-grading,
we need {\em gradings} of $M$ and $C_v$:
A {\em grading} of a symplectic manifold $(M, \omega)$
is the choice of a fiberwise universal cover
$
 \Lagtilde_M
$
of the Lagrangian Grassmannian bundle
$
 \Lag_M \to M.
$
A Lagrangian submanifold $L \subset M$ naturally gives a section
$$
\begin{array}{cccc}
 s_L : & L & \to & \Lag_M|_L \\
 & \vin & & \vin \\
 & x & \mapsto & T_x L,
\end{array}
$$
and a {\em grading} of $L$ is a choice of a lift of $s_L$ to
$$
 \stilde_L : L \to \Lagtilde_M|_L.
$$
When $M$ is equipped with a compatible almost complex structure
so that $(M, \omega, J)$ is an almost K\"{a}hler manifold,
the choice of a grading of $M$ is equivalent
to the choice of a nowhere-vanishing continuous section
$
 \eta
$
of the square
$
 (\Lambda^{\mathrm{top}} T^* M)^{\otimes 2}
$
of the canonical bundle.
Such a section exists
if and only if
$2 c_1(T M, J) = 0 \in H^2(M, \bZ)$,
and the homotopy classes of sections are classified
by $H^1(M, \bZ)$.
A section $\eta$ induces a map
$$
\begin{array}{cccc}
 {\det}^2_{\eta} : & \Lag_M & \to & \bCx / \bR^{>0} \cong S^1 \\
 & \vin & & \vin \\
 & \vspan \{ e_1, \dots, e_n \} & \mapsto
 & [\eta((e_1 \wedge \cdots \wedge e_n)^{\otimes 2})].
\end{array}
$$
The composition of $s_L$ and ${\det}_\eta^2$
will be denoted by $\phi_L$, and
a grading of $L$ is equivalent to
a lift $\phitilde_L : L \to \bR$
of $\phi_L : L \to S^1$ to the universal cover
$\bR \to S^1$.

Given a pair $(L_1, L_2)$ of graded Lagrangian submanifolds,
one can define the {\em Maslov index} $\mu(x; L_1, L_2)$
for each intersection point $x \in L_1 \cap L_2$.
If $\dim_{\bC} M = 1$,
then it is given by the round-up
$$
 \mu(x; L_1, L_2)
  = \lfloor \phitilde_{L_2}(x) - \phitilde_{L_1}(x) \rfloor
$$
of the difference of the phase function at $x$.

\begin{lemma} \label{lem:grading}
There is a grading of $(M, \omega)$ and $C_v$
such that the Maslov index of $e \in C_v \cap C_w$
for $v > w$ is given by
$$
 \mu(e; C_v, C_w) =
 \begin{cases}
   1 & e \nin D, \\
   2 & e \in D.
 \end{cases}
$$
\end{lemma}

\begin{proof}
Since $\dim_{\bR} M = 2$ in our case,
the Lagrangian Grassmannian bundle is the principal $S^1$-bundle
associated with $(T M)^{\otimes 2}$,
and a grading is a trivialization (i.e., a section) of it.
Let $\Gtilde$ be the pull-back of $G$ on $T = \bR^2 / \bZ^2$
to the universal cover $\bR^2$ of $T$, and
$\Mtilde$ be the 2-manifold associated with $\Gtilde$.
Fix a vertex $\vtilde_0$ of the quiver associated with $\Gtilde$ and
choose a grading of a tubular neighborhood of
$C_{\vtilde_0} \subset \Mtilde$
so that $C_{\vtilde_0}$ admits a grading.
Fix a grading of $C_{\vtilde_0}$ and
choose a grading of $C_\vtilde$ for other $\vtilde$
successively as follows;
if a vertex $\vtilde$ is adjacent by an arrow $a$
to another vertex $\vtilde'$
where the gradings of a tubular neighborhood of $C_{\vtilde'}$
and of $C_{\vtilde'}$ are already defined,
we choose a grading of a tubular neighborhood of $C_\vtilde$
and of $C_\vtilde$
so that the Maslov index $\mu(a; C_{\vtilde'}, C_\vtilde)$ is
one if $\vtilde' > \vtilde$, and minus one if $\vtilde' < \vtilde$.
Since these gradings on the tubular neighborhoods of
vanishing cycles glue together coherently
around a node of $\Gtilde$,
this defines gradings of $\Mtilde$ and $C_{\vtilde}$,
which descends to gradings of $M$ and $C_v$.
\end{proof}

A {\em relative grading} $\eta$
of an exact Lefschetz fibration $p : \scS \to \bC$
is a nowhere-vanishing section of the holomorphic line bundle
$
 \Lambda^{\mathrm{top}} (T^* \scS)^{\otimes 2}
  \otimes p^* (T^* \bC)^{\otimes (-2)},
$
which induces a grading on a regular fiber of $p$.
One can equip $p$ with a relative grading
such that the induced grading on $M$ coincides
with the one in Lemma \ref{lem:grading}
\cite[Section (16f)]{Seidel_PL}.

The following lemma concludes
the proof of Theorem \ref{th:fuk-A-infinity}:

\begin{lemma} \label{lem:sign}
Let $(e_0, e_1, \dots, e_k)$ be the sequence of edges of $G$
around a node respecting the cyclic order
such that $e_0 \in D$.
Then one has an $A_\infty$-operation
\begin{equation} \label{eq:A_infinity-operation}
 \m_k(e_k, \dots, e_1) = \pm e_0
\end{equation}
in the Fukaya category $\Fuk p$ of Lefschetz fibration,
where the sign is positive if the node is white,
and negative if it is black.
\end{lemma}

\begin{proof}
Since $M$ is a 2-manifold,
the $A_\infty$-operations in the Fukaya category are given
by counting polygons,
so that they are given as in \eqref{eq:A_infinity-operation} and
only the sign is the issue here.
We write the vanishing cycles surrounding
a polygon as $C_{v_i}$, $i = 0, \dots, k$
so that
$
 e_0 \in C_{v_0} \cap C_{v_k}
$
and
$
 e_i \in C_{v_{i-1}} \cap C_{v_i}
$
for
$
 i = 1, \dots, k.
$
The grading of $C_{v_i}$ defines an orientation of $C_{v_i}$,
and let $\xi_i$, $i = 1, \dots, k$ be
the unit tangent vector of $C_{v_i}$
at $e_i \in C_{v_{i-1}} \cap C_{v_i}$
along the orientation.
We also choose a point on each vanishing cycle $C_v$,
which comes from the choice of the non-trivial spin structure.
Then it follows from the Seidel's sign rule
\cite[Section (9e)]{Seidel_K3}
that the sign in \eqref{eq:A_infinity-operation}
is given by $(-1)^\dagger$,
where $\dagger$ is the sum
of the number of $i \in [1, k]$ such that
$\xi_i$ points away from $\varphi(D^2)$ and
the the number of points on $\varphi(\partial D^2)$
coming from the spin structures.

Now we use the following fact
from \cite{Ishii-Ueda_09}:
Let $G_0$ be the simplest dimer model
consisting of one black node,
one white node and three edges,
which corresponds to the McKay quiver
for the trivial group.
Then any consistent dimer model can be obtained from $G_0$
by successively adding a divalent node or an edge.

Assume that a dimer model $G$ is obtained
from another dimer model $G'$
by adding an edge between a black node $b$
and a white node $w$.
It follows from the definition of the grading of $M$ and $C_v$
that $\xi_i$ either points away from $\varphi(D^2)$
or toward $\varphi(D^2)$
simultaneously for all $i$,
and the fact that $b$ and $w$ are adjacent in $G$
implies that if $\xi_i$ at $b$ points away from the disk
corresponding to $b$,
then $\xi_i$ at $w$ points toward the disk
corresponding to $w$,
and if $\xi_i$ at $b$ points toward the disk
corresponding to $b$,
then $\xi_i$ at $w$ points away from the disk
corresponding to $w$.
Then either of the disk at $b$ or $w$ receives
an extra sign for $G$ compared to that for $G_0$,
which can be off-set by the introduction of a new base point
on the vanishing cycle,
which is necessary since one of the vanishing cycles for $G'$ is divided
into two by adding an edge to $G'$.
This shows that
the dimer model $G$ satisfies the assertion of the lemma
if the dimer model $G'$ does.
Since the addition of a divalent node
does not change $M$ and $C_v$ at all,
the lemma is proved.
\end{proof}

The {\em double suspension}
$
 \pss : \scS \times \bC^2 \to \bC
$
of
$
 p : \scS \to \bC
$
is defined by
$$
 \pss(x, u, v) = p(x) + u v,
$$
and a vanishing cycle $C \subset p(0)$ naturally gives
a vanishing cycle $L \subset (\pss)^{-1}(0)$
called the {\em double suspension} of $C$.
The following theorem is due to Seidel:

\begin{theorem}[{Seidel \cite[Corollary 5.5]{Seidel_suspension}}]
The full subcategory of $\Fuk ((\pss)^{-1}(0))$
consisting of double suspensions $L_v$ of $C_v$ for $v \in V$
is equivalent to the trivial extension of $\Fuk p$
of degree 3.
\end{theorem}

Since $\scA$ is the trivial extension of $\dirscA$ of degree 3
(or the {\em 3-dimensional cyclic completion}
in the terminology of \cite{Segal_ADT}),
Corollary \ref{cor:local_hms} is proved.

\bibliographystyle{alpha}
\bibliography{bibs}

\noindent
Masahiro Futaki

Graduate School of Mathematical Sciences,
The University of Tokyo,
3-8-1 Komaba Meguro-ku Tokyo 153-8914, Japan

{\em e-mail address}\ : \  futaki@ms.u-tokyo.ac.jp

\ \\

\noindent
Kazushi Ueda

Department of Mathematics,
Graduate School of Science,
Osaka University,
Machikaneyama 1-1,
Toyonaka,
Osaka,
560-0043,
Japan.

{\em e-mail address}\ : \  kazushi@math.sci.osaka-u.ac.jp
\ \vspace{0mm} \\

\end{document}